\documentclass[oneside,a4paper,12pt]{amsart} 
\usepackage{amssymb}
\usepackage{amscd}
\usepackage[all,cmtip]{xy}
\usepackage{verbatim}
\usepackage{lscape}
\usepackage{enumerate}
\usepackage{xcolor}
\theoremstyle{plain}
\newtheorem{theorem}{Theorem}[section]

\newtheorem{corollary}[theorem]{Corollary}

\theoremstyle{definition}

\theoremstyle{remark}

 \DeclareMathOperator{\Hom}{Hom}

\DeclareMathOperator{\invlim}{\varprojlim}
\DeclareMathOperator{\dirlim}{\varinjlim}
\DeclareMathOperator{\ide}{id}

\begin{document}

\title{Cohomology of Profinite Groups of Bounded Rank}

\author{Peter Symonds}
\address{Department of Mathematics\\
         University of Manchester\\
     Manchester M13 9PL\\
     United Kingdom}
\email{Peter.Symonds@manchester.ac.uk}

\subjclass[2010]{Primary: 20J06; Secondary: 20E18}

\begin{abstract}
We generalise to profinite groups some of our previous results on the cohomology of pro-$p$ groups of bounded sectional $p$-rank.
\end{abstract}

\maketitle

\section{Introduction}
\label{sec:intro}

The purpose of this note is to generalise to profinite groups the results of \cite{Sy} for pro-$p$ groups of bounded rank. The principal one of these states that amongst the pro-$p$ groups of rank bounded by a number $r$ there are only finitely many mod-$p$ cohomology rings up to isomorphism. Recall that rank here means the $p$-sectional rank, which is the maximum of the ranks of $H/K$ where $H \leq G$ and $K \trianglelefteq H$ and $H/K$ is an elementary abelian $p$-group. The generalisation is as follows.

\begin{theorem}
	\label{th:main}
	For given $p$ and $r$, the profinite groups of $p$-sectional rank at most $r$ have only finitely many graded isomorphism classes of mod-$p$ cohomology rings between them and these are all noetherian.
\end{theorem}

Another result bounds the dimensions of the cohomology groups.

\begin{theorem}
	\label{th:bounds}
	If $G$ is a profinite group of $p$-sectional rank at most $r$ then:
	\begin{enumerate}
		\item
		 $\dim_{\mathbb F_p} H^i(G; \mathbb F_p) \leq \binom{r(\lceil \log_2r \rceil  +3 +e)+i-1}{i}$, where $e=0$ for $p$ odd and $e=1$ for $p=2$; 
	\item
	there is a function $X(p,r)$ such that  $\dim_{\mathbb F_p} H^i(G; \mathbb F_p) \leq X(p,r) \cdot i^{a-1}$, where $a$ is the maximum rank of an elementary abelian $p$-subgroup.
	\end{enumerate}
\end{theorem}

In Section~\ref{se:fus} we will generalise these results to pro-fusion systems. 

The generalisation from pro-$p$ to profinite is harder than in the finite case, because a Sylow pro-$p$ subgroup might have infinite index.

\section{Proofs}
\label{sec:proofs}

We fix a prime $p$; all cohomology groups will have coefficients in $\mathbb F_p$ and rank will mean $p$-sectional rank. A homomorphism of cohomology rings means an homomorphism of graded rings. A subgroup will always mean a subgroup in the category of profinite groups.

First we prove part (1) of Theorem~\ref{th:bounds}. This is proved for finite $p$-groups in \cite[1.2]{Sy}. If $G$ is a finite group with Sylow $p$-subgroup $S$ then a standard transfer argument shows that the restriction map embeds $H^*(G)$ in $H^*(S)$, so the bound hold for $G$.

If $G$ is profinite, say $G=\invlim G/N_i$ with the $G/N_i$ finite, then $H^*(G)=\dirlim H^*(G/N_i)$. The ranks of the $G/N_i$ are also bounded by $r$, so the bound applies to each $H^*(G/N_i)$ and hence to their direct limit.

Now we prove Theorem~\ref{th:main}. This is proved for finite $p$-groups in \cite[1.1]{Sy} and for pro-$p$ groups in \cite[1.4]{Sy}. Let $G$ be a finite group with Sylow $p$-subgroup $S$. Restriction embeds $H^*(G)$ in $H^*(S)$ and the image can be characterised as the subring of stable elements, see e.g.\ \cite[XII 10.1]{CE}.

The stable elements of $H^*(S)$ are the $x \in H^*(S)$ such that, for any inclusion of a subgroup $i_P \! : P \rightarrow S$ and any homomorphism $\varphi \! : P \rightarrow S$ induced by inclusion and then conjugation by an element of $G$, we have $(i_P^*-\varphi^*)(x)=0$. The usual formulation only considers certain subgroups $P$, but all the conditions are certainly necessary on the image of $H^*(G)$, so this formulation is also valid.

By the validity of the result for finite $p$-groups, there are only finitely many possible isomorphism classes of rings $H^*(S)$ and $H^*(P)$. Since $H^*(S)$ is noetherian, there are only finitely many graded ring homomorphisms $H^*(S) \rightarrow H^*(P)$. Thus there are only finitely many different conditions of the form $(i_P^*-\varphi^*)(x)=0$ and hence only a finite number of possible subrings of stable elements. The cohomology of any finite group is noetherian. This completes the case of a finite group.

Since there are only finitely many possible cohomology rings for a finite group of rank at most $r$ and they are all finitely generated, there is a number $N$, depending only on $p$ and $r$, such that they are all generated in degrees at most $N$. The dimension of the sum of the cohomology groups in degrees 0 through $N$ is bounded in terms of $r$ and $N$, by Theorem~\ref{th:bounds}(1), and this bounds the number of generators needed. The proof of \cite[1.3]{Sy} now applies verbatim to prove part (2) of \ref{th:bounds}.

Let $G$ be a profinite group of rank at most $r$, say $G=\invlim G/N_i$, so $H^*(G)=\dirlim H^*(G/N_i)$. Thus $H^*(G)$ will also be generated in degrees at most $N$. Let $S$ be the Sylow pro-$p$ subgroup of $G$; then the restriction map identifies $H^*(G)$ with its image under restriction to $H^*(S)$ (because it does so on each finite quotient). By the pro-$p$ case of the theorem, there are only finitely many possible rings $H^*(S)$ and the part of $H^*(S)$ in degrees at most $N$ is finite. It follows that there are only finitely many subrings generated in degrees at most $N$ and these subrings are all finitely generated and so noetherian. This completes the proofs of Theorems \ref{th:main} and \ref{th:bounds}.

\section{Inflation Functors and Pro-Fusion Systems}
\label{se:fus}

An inflation functor is defined in \cite{Webb} (it is called a functor with Mackey structure in \cite{Sy2}). An inflation functor $M$ is a contravariant functor $M^*$ from the category of finite groups to the category of $R$-modules for some ring $R$ that is also a covariant functor $M_*$ on the subcategory of finite groups and injective group homomorphisms. The two structures are related, in particular the Mackey double coset formula holds. Examples are cohomology $H^*(G;R)$, where the covariant part is given by transfer and the representation or Green rings, where it is given by induction.

Such a functor $M$ can be extended to profinite groups by setting $M(G)=\dirlim M^*(G/N_i)$, when $G = \invlim G/N_i$. This is well defined as a contravariant functor on the category of profinite groups (and continuous homomorphisms) and the covariant structure is defined on injective homomorphisms with open image. 
The Mackey formula still holds. All of this is familiar in the case of cohomology, where the covariant structure is given by the transfer. We can also use the same construction on just a contravariant functor on finite groups or finite $p$-groups to obtain a functor on profinite groups or pro-$p$ groups respectively.

An inflation functor is said to be cohomological if $M_*(i_H) \circ M^*(i_H) = |G:H| \ide_{M(G)}$ for $i_H$ the inclusion of a subgroup $H$ in $G$. If this holds for all finite groups then it holds whenever $G$ is profinite and $H$ is an open subgroup. Cohomology $H^*(-; \mathbb F_p)$  is a cohomological inflation functor; the usual definition on a profinite group also satisfies the colimit property above. 

Pro-fusion systems were introduced in \cite{SS} as a generalisation of fusion systems to pro-$p$ groups. We will freely refer to that paper and use its notation without comment.  Roughly speaking, a pro-fusion system $\mathcal F$ on a pro-$p$ group $S$ is an inverse limit of fusion systems $\mathcal F_i$ on certain finite quotients $S_i$ of $S$.

For any pro-fusion system $\mathcal F$ on a pro-$p$ group $S$ and any inflation functor $M$ we define the stable elements $M(S)^{\mathcal F}$ to be the submodule of elements $x \in M(S)$ such that, for any inclusion of a subgroup $i_P \! : P \rightarrow S$ and any homomorphism $\varphi \in \Hom_{\mathcal F}( P, S)$, we have $(i_P^*-\varphi^*)(x)=0$ (where we write $f^*$ for $M^*(f)$). The method of stable elements, as mentioned above, is usually stated for cohomology, but the proof applies whenever $G$ is a finite group with Sylow $p$-subgroup $S$, $M$ is a cohomological inflation functor (or just a cohomological Mackey functor on $G$) and every rational prime except $p$ is invertible in $R$ to show that restriction $M(G) \rightarrow M(S)$ is injective with image $M(S)^{\mathcal F_S(G)}$, where $\mathcal F_S(G)$ is the fusion system on $S$ induced by $G$. Our aim is to extend this to profinite groups.

If $G$ is a profinite group with Sylow $p$-subgroup $S$, say $G = \invlim G/N_i$, then $M(G)=\dirlim M^*(G/N_i) \cong \dirlim M(S_i)^{\mathcal F_{S_i}(G/N_i)}$, where $S_i=SN_i/N_i$. There is also a pro-fusion system $\mathcal F_S(G)=\invlim \mathcal F_{S_i}(G/N_i)$ on $S$, which is pro-saturated. Thus what we need is the next result, which has also been considered in \cite{DGM}. I am grateful to the authors of \cite{DGM} for pointing out an error in my original proof.

\begin{theorem}
	\label{th:stable}
	Let $M$ be a contravariant functor from $p$-groups to $R$-modules, extended to pro-$p$ groups as above. Let $S$ be a pro-$p$ group and let $\mathcal F = \invlim \mathcal F_i$ be a pro-saturated pro-fusion system on $S$. Then inflation induces an isomorphism
	\[ M(S)^{\mathcal F} \cong \dirlim M(S_i)^{\mathcal F_i}. \]
	\end{theorem}

\begin{proof}
	It is straightforward to show that the inflation maps induce an embedding of $\dirlim M(S_i)^{\mathcal F_i}$ into $M(S)^{\mathcal F}$, so we concentrate on proving that this map is surjective. We use notation from \cite{SS}. 
	
	Because $\mathcal F$ is pro-saturated, we know from \cite[4.5]{SS} that we can take the $\mathcal F_i$ to be the saturated fusion systems $\mathcal F/N$ as $N$ runs through the open strongly $\mathcal F$-closed subgroups of $S$. It follows that for  any $P \leq S_i$, $\varphi \in \Hom_{\mathcal F_i}(P,S_i)$ and $j \geq i$ there is a $\tilde{\varphi}_j \in \Hom_{\mathcal F_j}(\tilde{P}_j, S_j)$ such that $f_{i,j} \tilde{\varphi}_j=\varphi f_{i,j} {\big | _{\tilde{P}_j}}$, where $f_{j,k} \! : S_k \rightarrow S_j$ is the quotient map and $\tilde{P}_j=f_{i,j}^{-1}(P)$. Similarly, $\varphi$ can also be lifted to $\tilde{\varphi} \in \Hom_{\mathcal F}(\tilde{P}, S)$.
	
	Let $x \in M(S)^{\mathcal F}$; then $x=f_j^*(x_j)$ for some $j$ and some $x_j \in M(S_j)$, where $f_j \! : S \rightarrow S_j$ is the quotient map. The element $x$ satisfies the conditions $(i_P-\varphi)^*(x)=0$ for all $P \leq S$ and $\varphi \in \Hom_{\mathcal F}(P,S)$. 
	
	Let $Q \leq S_j$ and $\theta \in \Hom_{\mathcal F_j}(Q,S_j)$. Then 
	\begin{multline*} {f_j \big | _{\tilde{Q}}}^*(i_Q^*-\theta^*)(x_j) =(i_{Q}f_j \big | _{\tilde{Q}}-\theta f_j \big | _{\tilde{Q}})^*(x_j) \\
		=(f_j i_{\tilde{Q}}-f_j \tilde{\theta})^*(x_j)
		=(i_{\tilde{Q}}-\tilde{\theta})^*f_j^*(x_j)
		=(i_{\tilde{Q}}-\tilde{\theta})^*(x)
		=0.\end{multline*}
	 Thus there is an $\ell=\ell(Q,\theta) \geq j$ such that $ f_{j,\ell}\big | _{\tilde{Q}_\ell}^*(i_Q^*-\theta^*)(x_j) =0$. There are only finitely many possible different $Q$ and $\theta$ so there is a $k$ such that $k \geq \ell(Q,\theta)$ for all of them and hence   $(i_{\tilde{Q}_k}-\tilde{\theta}_k^*)f_{j,k}^*(x_j)=f_{j,k}\big | _{\tilde{Q}_k}^*(i_Q^*-\theta^*)(x_j)=0$ for all $Q$ and $\theta$. 

Set $x_k=f_{j,k}^*(x_j)\in M(S_k)$; we need to show that $x_k\in M(S_k)^{\mathcal F _k}$. Suppose that $P \leq S_k$ and $\varphi \in \Hom_{\mathcal F_k}(P,S_k)$. Let $Q=f_{j,k}(P)$; then $\varphi$ induces a $\theta \in \Hom_{\mathcal F_j}(Q,S_j)$ such that $f_{j,k} \varphi = \theta f_{j,k} \big |_P$. As mentioned at the beginning of the proof, there is a $\hat{\theta} \in \Hom _{\mathcal F_k} (\hat{Q}_k,S_k)$, where $\hat{Q}_k= f_{j,k}^{-1}(Q)$, such that $f_{j,k} \hat{\theta}_k = \theta f_{j,k} \big |_{\hat{Q}_k}$. Thus $f_{j,k}\hat{\theta}_k i_P^{\hat{Q}_k} = f_{j,k} \varphi$, where $i_P^{\hat{Q}_k}$ is the inclusion of $P$ in $\hat{Q}_k$. Now
	\begin{multline*} (i_P-\varphi)^*(x_k)
	=(i_P-\varphi)^*f_{j,k}^*(x_j)
	=(f_{j,k}i_P-f_{j,k} \varphi)^*(x_j) \\
	=(f_{j,k}i_{\hat{Q}_k} i^{\hat{Q}_k}_P- f_{j,k} \hat{\theta}_k i^{\hat{Q}_k}_P)^*(x_j)
	=i^{\hat{Q}_k*}_P(i_{\hat{Q}_k}-\hat{\theta}_k)^*f_{j,k}^*(x_j)
	=0, \end{multline*} 
as required.

\end{proof}

As was pointed out above, we can now deduce the next result.

\begin{corollary}
	\label{co:stab}
	Let $M$ be a cohomological inflation functor over a ring $R$ in which every rational prime except $p$ is invertible, extended to profinite groups as above. Let $G$ be a profinite group; then $M(G) \cong M(S)^{\mathcal F_S(G)}$.
	\end{corollary}

This motivates the study of $M(S)^{\mathcal F}$ for an arbitrary pro-fusion system. There is also a dual version of this theory that applies to homology.
	
\begin{theorem}
	\label{th:fus}
	For given $p$ and $r$, consider the rings $H^*(S)^{\mathcal F}$, where $S$ is a pro-$p$ group with $p$-sectional rank at most $r$ and $\mathcal F$ is a pro-saturated pro-fusion system on $S$.
		Then there are only finitely many such rings up to isomorphism
		and they are all noetherian.
	\end{theorem}

\begin{proof} By hypothesis, $\mathcal F$ can be expressed as an inverse limit of fusion systems $\mathcal F_i$ on finite $p$-groups $S_i$. By Theorem~\ref{th:stable}, $H^*(S)^{\mathcal F}$ is the direct limit of the $H^*(S_i)^{\mathcal F_i}$. The finiteness of the number of rings is now proved in the same way as in the case of finite groups in Theorem~\ref{th:main}.
	
	 For the noetherian property we use the fact from \cite[5.2]{BLO} that each $H^*(S_i)^{\mathcal F_i}$ is noetherian. We have just seen that there are only finitely many isomorphism classes of such rings, thus there is a number $N$ such that they are all generated in degrees at most $N$. This property passes to the direct limit, which is $H^*(S)^{\mathcal F}$.
	
	We are therefore considering subrings of $H^*(S)$ that are generated in degrees at most $N$. But $H^*(S)$ is finite in this range of degrees, so the subrings must be noetherian.
		\end{proof}

Of course, Theorem~\ref{th:main} can be seen as a corollary of these last two results.

\section{The Steenrod Algebra}	

In any situation where we have only finitely many non-isomorphic cohomology rings and they are all noetherian, we also have only finitely many non-isomorphic algebras over the Steenrod algebra. This is because, on any particular ring, the action of a Steenrod power is determined by the images of the generators of the ring, by the Cartan formula, and there are only finitely many possibilities. Also, sufficiently high powers act trivially, because the noetherian assumption implies that there is a bound the degrees of the generators, so these are sent to 0, by the condition for being unstable.

If we take cohomology with coefficients in the $p$-adic integers instead of $\mathbb F_p$ we get  infinitely many different cohomology rings.
This happens even in the case of sectional $p$-rank 1, which includes the cyclic groups of $p$-power order.


\end{document}